\documentclass[10pt]{amsart}

\usepackage{latexsym}
\usepackage{amsfonts}

\usepackage{amsmath,amsthm,amssymb}

\author[I.~Kapovich]{Ilya Kapovich}

\address{\tt Department of Mathematics, University of Illinois at
  Urbana-Champaign, 1409 West Green Street, Urbana, IL 61801, USA
  \newline http://www.math.uiuc.edu/\~{}kapovich/} \email{\tt
  kapovich@math.uiuc.edu}

\title[Algorithmic detectability of iwip automorphisms]{Algorithmic detectability of iwip automorphisms}

\newtheorem{theor}{Theorem}

\newtheorem{thm}{Theorem}[section] \newtheorem{lem}[thm]{Lemma}
\newtheorem{cor}[thm]{Corollary} 
\newtheorem{prop}[thm]{Proposition} \theoremstyle{definition}
\newtheorem{defn}[thm]{Definition}

 \newtheorem{rem}[thm]{Remark}

\def\strutdepth{\dp\strutbox}
\def \ss{\strut\vadjust{\kern-\strutdepth \sss}}
\def \sss{\vtop to \strutdepth{
\baselineskip\strutdepth\vss\llap{$\diamondsuit\;\;$}\null}}

\def\strutdepth{\dp\strutbox}
\def \sst{\strut\vadjust{\kern-\strutdepth \ssss}}
\def \ssss{\vtop to \strutdepth{
\baselineskip\strutdepth\vss\llap{$\spadesuit\;\;$}\null}}

\def\strutdepth{\dp\strutbox}
\def \ssh{\strut\vadjust{\kern-\strutdepth \sssh}}
\def \sssh{\vtop to \strutdepth{
\baselineskip\strutdepth\vss\llap{$\heartsuit\;\;$}\null}}

\def\epsilon{\varepsilon}
\def\phi{\varphi}

\newcommand{\Out}{\mbox{Out}}
\newcommand{\Aut}{\mbox{Aut}}

\begin{document}

\begin{abstract}
We produce an algorithm that, given $\phi\in \Out(F_N)$, where $N\ge 2$, decides wether or not $\phi$ is an iwip ("fully irreducible") automorphism. 
\end{abstract}

\thanks{The author was partially supported by the NSF
  grant DMS-0904200}


\maketitle

\section{Introduction}

The notion of a pseudo-Anosov homeomorphism of a compact surface plays a fundamental role in low-dimensional topology and the study of mapping class groups. In the context of $\Out(F_N)$ the concept of being pseudo-Anosov has several (non-equivalent) analogs.

The first is the notion of an "atoroidal" automorphism. An element $\phi\in \Out(F_N)$ is called \emph{atoroidal} if there do not exist $m\ge 1$, $h\in F_N, h\ne 1$ such that $\phi^m$ preserves the conjugacy class $[h]$ of $h$ in $F_N$. A key result of Brinkmann~\cite{Br}, utilizing the Bestvina-Feighn Combination Theorem~\cite{BF92}, says that $\phi\in \Out(F_N)$ is atoroidal if and only if the mapping torus group of some (equivalently, any) representative $\Phi\in \Aut(F_N)$ of $\phi$ is word-hyperbolic. 
Another, more important, free group analog of being pseudo-Anosov is the notion of a "fully irreducible" or "iwip" automorphism. An element $\phi\in \Out(F_N)$ is called \emph{reducible} if there exists a free product decomposition $F_N=A_1\ast \dots \ast A_k\ast C$ with $k\ge 1$, $A_i\ne 1$ and $A_i\ne F_N$ such that $\phi$ permutes the conjugacy classes $[A_1], \dots, [A_k]$. An element $\phi\in \Out(F_N)$ is \emph{irreducible} if it is not reducible. An element $\phi\in \Out(F_N)$ is \emph{fully irreducible} or \emph{iwip} (which stands for "irreducible with irreducible powers") if $\phi^m$ is irreducible for all integers $m\ge 1$ (equivalently, for all nonzero integers $m$). Thus $\phi$ is an iwip if and only if there do not exist a proper free factor $A$ of $F_N$ and $m\ge 1$ such that $\phi^m([A])=[A]$. The notion of an iwip automorphism plays a key role in the study of geometry and dynamics of $\Out(F_N)$ and of the Culler-Vogtmann Outer space (see, for example \cite{LL,Vog,BFH97,Gui1,BG,CP,CH12,KL4}, etc).

If $S$ is a connected compact surface, there are well-known algorithms (e.g. see \cite{BH95}) to decide whether or not an element $g\in Mod(S)$ of the mapping class group of $S$ is pseudo-Anosov. Similarly, because of the result of Brinkmann mentioned above, it is easy (at least in theory) to decide algorithmically whether an element $\phi\in \Out(F_N)$ is atoroidal. Namely, we pick a representative  $\Phi\in \Aut(F_N)$ of $\phi$, form the mapping torus group $G=F_N\rtimes_\Phi \mathbb Z$ of $\Phi$ and start, in parallel, checking if $G$ is hyperbolic (e.g. using the partial algorithm of Papasoglu~\cite{Pa} for detecting hyperbolicity) while at the same time looking for periodic conjugacy classes of nontrivial elements of $F_N$. Eventually exactly one of these procedures will terminate and we will know whether or not $\phi$ is atoroidal. A similar algorithm can be used to decide, for a closed hyperbolic surface $S$, if an element $g\in Mod(S)$ is pseudo-Anosov.

By contrast, there is no obvious approach to algorithmically deciding whether an element $\phi\in \Out(F_N)$ is an iwip. 
In this note we provide such an algorithm:

\begin{theor}\label{thm:A}
There exists an algorithm that, given $N\ge 2$ and $\phi\in \Out(F_N)$ decides whether or not $\phi$ is an iwip.
\end{theor}

A key step in the argument is an "if and only if" criterion of iwipness for atoroidal elements of $\Out(F_N)$ in terms of Whitehead graphs of train-track representatives of $\phi$, see Proposition~\ref{prop:crit} below. Proposition~\ref{prop:crit} is similar to and inspired by Lemma~9.9 in a recent paper of Pfaff~\cite{Pf}; see also Proposition~5.1 in a paper of J\"ager and Lustig~\cite{JL} for a related criterion of iwipness.
Compared to the proof of Lemma~9.9 in \cite{Pf}, our proof of Proposition~\ref{prop:crit} is more elementary  and does not involve any relative train-track technology or any machinery from the Bestvina-Feign-Handel work~\cite{BFH00} on the Tits Alternative for $\Out(F_N)$.
However, we do utilize the notion of a "stable lamination" developed by Bestvina-Feign-Handel in \cite{BFH97} for iwip elements of $\Out(F_N)$.

To the best of our knowledge, the statement of Theorem~\ref{thm:A} does not exist in the literature, although it is most likely that this result is known to some experts in the field.
Since the notion of an iwip plays such a fundamental role in the study of $\Out(F_N)$, we think it is useful to put a proof of Theorem~\ref{thm:A} in writing.

In a subsequent paper of the author with Dowdall and Leininger~\cite{DKL}, the conclusion of Proposition~\ref{prop:crit}  was improved by showing that the assumption in  Proposition~\ref{prop:crit} that there exists a positive power $f^k$ of $f$ with $A(f)>0$ may be replaced by the assumption that $A(f)$ be irreducible. See Proposition~\ref{prop:DKL} below for a precise statement. This fact, together with Proposition~\ref{prop:crit}, was used in \cite{DKL} to show that for an atoroidal $\phi\in \Out(F_N)$ being irreducible is equivalent to being an iwip; see Corollary~\ref{cor:DKL} below.

I am grateful to Martin Lustig for useful discussions and to Matt Clay for pointing out and correcting an error in the enumeration procedure in Case~2 of Theorem~\ref{thm:main} in an earlier version of this paper. I am also grateful to Chris Leininger for suggesting a slight simplification of the proof of Proposition~\ref{prop:BFH97}. Finally, I wish to thank the referee for the unusually positive tone of the report.

\section{Train-track and graph terminology}

For a free group $F_N$ (where $N\ge 2$) we fix an identification $F_N=\pi_1(R_N)$, where $R_N$ is the \emph{$N$-rose}, that is, a wedge of $N$ circles. 

We will only briefly recall the basic definitions related to train-tracks for free group automorphisms. We refer the reader to \cite{BH92,DV,Bo,BFH00,BG} for detailed background information.

\subsection{Graphs and graph-maps}

By a \emph{graph} we mean a 1-dimensional cell-complex. For a graph $\Gamma$ we refer to $0$-cells of $\Gamma$ as \emph{vertices} and to open $1$-cels of $\Gamma$ as \emph{topological edges}. We denote the set of vertices of $\Gamma$ by $V\Gamma$ and the set of topological edges of $\Gamma$ by $E_{top}\Gamma$. Each topological edge of $\Gamma$ is homeomorphic to $(0,1)$ and thus admits exactly two orientations. A topological edge with a choice of an orientation is called an \emph{oriented edge} or just \emph{edge} of $\Gamma$. We denote the set of oriented edges of $\Gamma$ by $E\Gamma$. For an oriented edge $e$ of $\Gamma$ we denote by $o(e)$ the initial vertex of $e$ and by $t(e)$ the terminal vertex of $e$; we also denote by $e^{-1}$ the edge $e$ with the opposite orientation. Thus $o(e^{-1})=t(e), t(e^{-1})=o(e)$ and $(e^{-1})^{-1}=e$.

If $\Gamma$ is a graph, a \emph{turn} in $\Gamma$ is an unordered pair $e,e'$ of oriented edges of $\Gamma$ such that $o(e)=o(e')$. A turn $e,e'$ is \emph{degenerate} if $e=e'$ and \emph{non-degenerate} if $e\ne e'$. 

An \emph{edge-path} in a graph $\Gamma$ is a sequence $\gamma=e_1,\dots, e_n$ of $n\ge 1$ oriented edges such that $t(e_i)=o(e_{i+1})$ for all $1\le i<n$. We say that $n$ is the \emph{simplicial length} of $\gamma$ and denote $|\gamma|=n$.  We put $o(\gamma):=o(e_1)$, $t(\gamma):=t(e_n)$ and $\gamma^{-1}:=e_n^{-1},\dots, e_1^{-1}$.
We also view a vertex $v$ of $\Gamma$ as an edge-path $\gamma$ of simplicial length $0$ with $o(\gamma)=t(\gamma)=v$.

If $\gamma=e_1,\dots, e_n$ is an edge-path in $\Gamma$ and $e,e'$ is a turn in $\Gamma$, we say that this turn is \emph{contained in $\gamma$} if there exists $1\le i<n$ such that $e_i=e^{-1},e_{i+1}=e'$ or $e_i=(e')^{-1},e_{i+1}=e$.

An edge-path $\gamma=e_1,\dots, e_n$ is \emph{tight} or \emph{reduced} if there does not exist $i$ such that $e_{i+1}=e_i^{-1}$, that is, if every turn contained in $\gamma$ is non-degenerate. A closed edge-path $\gamma=e_1,\dots, e_n$ is \emph{cyclically tight} or \emph{cyclically reduced} if every cyclic permutation of $\gamma$ is tight.

%For an edge-path $\gamma$ we denote the tightened (rel endpoints) form of $\gamma$ by $[\gamma]$.
%For a closed edge-path $\gamma$ we denote the cyclically tightened form of $\gamma$ by $[[\gamma]]$.

If $\Gamma_1,\Gamma_2$ are graphs, a \emph{graph-map} is a continuous map $f:\Gamma_1\to\Gamma_2$ such that $f(V\Gamma_1)\subseteq V\Gamma_2$ and such that for every oriented edge $e$ of $\Gamma_1$ its image $f(e)=e_1,\dots, e_n$ is a tight edge-path of positive simplicial length.
More precisely, we mean that there exists a finite subdivision $x_0=o(e), x_1,\dots, x_n=t(e)$ of $e$ such that $f(x_i)=t(e_i)$ for $i=1,\dots, n$ and that for $i=1,\dots, n$ the continuous map $f$ maps the open interval of $e$ between $x_{i-1}$ and $x_i$ homeomorphically onto the open edge $e_i$.  When graph-maps and train-track maps are defined in the context of studying $\Out(F_N)$, one often requires the graphs $\Gamma_1$ and $\Gamma_2$ to come equipped with specific chosen PL-structures and the graph-maps to respect those structures. See \cite{DKL} for a careful discussion on the topic. However, in the present paper we do not need these extra assumptions and, in the terminology of \cite{DKL}, we work with "topological graphs" and "topological graph-maps". 

Every graph-map $f:\Gamma_1\to\Gamma_2$ comes equipped with its \emph{derivative map} $Df: E\Gamma_1\to E\Gamma_2$: for each $e\in E\Gamma_1$ we define $(Df)(e)$ to be the initial edge of $f(e)$. 

Let $\Gamma$ be a finite graph and let $f:\Gamma\to\Gamma$ be a graph-map. Let $r=\#E_{top}\Gamma$ and let $E_{top}\Gamma=\{e_1',\dots, e_r'\}$ be an ordering of the set of topological edges of $\Gamma$. For each $i=1,\dots, r$ let $e_i$ be an oriented edge corresponding to some choice of an orientation on the topological edge $e_i'$.
The \emph{transition matrix} $A(f)=(a_{ij})_{i,j=1}^r$ of $f$ (with respect to this ordering) is an $r\times r$-matrix where the entry $a_{ij}$ is the total number of occurrences of $e_i^{\pm 1}$ in the path $f(e_j)$. We say that $A(f)$ is \emph{positive}, denoted $A(f)>0$, if $a_{ij}>0$ for all $1\le i,j\le r$. We say that $A=A(f)$ is \emph{irreducible} if for every $1\le i,j\le r$ there exists $t=t(i,j) \ge 1$ such that $(A^t)_{ij}>0$. Thus if $A(f)>0$ then $A(f)$ is irreducible.
Recall that a vertex $v\in V\Gamma$ is $f$-\emph{periodic}  (or just \emph{periodic}) if there exists $n\ge 1$ such that $f^n(v)=v$. Similarly, an edge $e\in E\Gamma$ is $f$-\emph{periodic} (or just \emph{periodic})  if there exists $n\ge 1$ such that $f^n(e)$ starts with $e$.  Since the sets $V\Gamma$ and $E\Gamma$ are finite, periodic vertices and periodic edges always exist.

\subsection{Train-tracks}

Let $\Gamma$ be a finite connected graph. A graph-map $f:\Gamma\to\Gamma$ is a \emph{train-track map} if for every edge $e\in E\Gamma$ and for every $n\ge 1$ the path $f^n(e)$ is tight (that is, if all the turns contained in $f^n(e)$ are non-degenerate). 
A train-track map $f:\Gamma\to\Gamma$  is \emph{expanding} if there exists $e\in E\Gamma$ such that $|f^n(e)|\to\infty$ as $n\to\infty$.

\begin{rem}
If $f:\Gamma\to\Gamma$ is a train-track map, then for every $m\ge 1$ the map $f^m:\Gamma\to\Gamma$ is also a train-track map. Moreover, the definition of  the transition matrix implies that for every $m\ge 1$ we have $A(f^m)=[A(f)]^m$.
\end{rem}

If $f:\Gamma\to\Gamma$ is a train-track map, we say that a turn $e,e'$ in $\Gamma$ is \emph{taken} by $f$ is there exist $n\ge 1$ and $e''\in E\Gamma$ such that the turn $e,e'$ is contained in the path $f^n(e'')$. Note that a taken turn is necessarily non-degenerate, since $f$ is a train-track map.

Let $\phi\in \Out(F_N)$. A \emph{topological representative}  of $\phi$ consists of a homotopy equivalence $\alpha:R_N\to \Gamma$ (sometimes called a \emph{marking}), where $\Gamma$ is a finite connected graph, and a graph-map $f:\Gamma\to\Gamma$ with the following properties:
\begin{enumerate}
\item The map $f$ is a homotopy equivalence.

\item If $\beta:\Gamma\to R_N$ is a homotopy inverse of $\alpha$ then at the level of $F_N=\pi_1(R_N)$, the map $\beta\circ f\circ \alpha: R_N\to R_N$ induces precisely the outer automorphism $\phi$.
\end{enumerate}

If $f:\Gamma\to\Gamma$ is a topological representative of $\phi\in \Out(F_N)$, a subgraph $\Delta\subseteq \Gamma$ is called a \emph{reduction for $\phi$} if $\Delta$ is $f$-invariant (that is $f(\Delta)\subseteq \Delta$), the inclusion $\iota:\Delta\to\Gamma$ is not a homotopy equivalence and if $\Delta$ is homotopically nontrivial, that is, at least one connected component of $\Delta$ is not contractible.
As shown in \cite{BH92}, $\phi\in\Out(F_N)$ is irreducible if and only if no topological representative of $\phi$ admits a reduction. In particular, if $f$ is a topological representative of $\phi$ which admits a reduction, then $\phi$ is reducible. We will use this fact in the proof of  Proposition~\ref{prop:BH92} below.

If $f:\Gamma\to\Gamma$ is a topological representative of $\phi\in \Out(F_N)$, the fact that $f$ induces a quasi-isometry of the universal cover $\tilde\Gamma$ of $\Gamma$ implies that for any semi-infinite tight edge-path $\rho=e_1,e_2,\dots$ in $\Gamma$ the path $f(\rho)=f(e_1)f(e_2)\dots$ tightens to a unique tight semi-infinite edge-path $\rho'$ starting with the vertex $o(f(e_1))$. 

For an outer automorphism $\phi\in \Out(F_N)$ (where $N\ge 2$), a \emph{train-track representative} of $\phi$ is a topological representative $f:\Gamma\to\Gamma$ of $\phi$ such that $f$ is a train-track map, and such that every vertex of $\Gamma$ has degree $\ge 3$.  Note that if $f:\Gamma\to\Gamma$ is a train-track representative of $\phi$ then for every $m\ge 1$ the map  $f^m:\Gamma\to\Gamma$ is a train-track representative of $\phi^m$. 

An important basic result of Bestvina and Handel~\cite{BH92} states that every irreducible $\phi\in \Out(F_N)$ (where $N\ge 2$) admits a train-track representative with an irreducible transition-matrix.

\begin{defn}[Whitehead graph of a train-track]
Let $f:\Gamma\to\Gamma$ be a train-track map representing $\phi\in \Out(F_N)$. Let $v\in V\Gamma$. The \emph{Whitehead graph} $Wh_\Gamma(v,f)$ is a simple graph defined as follows. The set of vertices of $Wh_\Gamma(v,f)$ is the set of all oriented edges $e$ of $\Gamma$ with $o(e)=v$.

Two distinct oriented edges $e',e''$ of $\Gamma$ with origin $v$ represent adjacent vertices in $Wh_\Gamma(v,f)$ if the turn $e',e''$ is taken by $f$, that is, if there exist $e\in E\Gamma$ and $n\ge 1$ such that the turn $e', e''$ is contained in the edge-path $f^n(e)$.

\end{defn}

\begin{rem}\label{rem:positive}
Let $\phi\in \Out(F_N)$ (where $N\ge 2$) and let $f:\Gamma\to\Gamma$ be a train-track representative such that for some $m\ge 1$ we have $A(f^m)>0$.
Then $A(f^t)$ is irreducible for all $t\ge 1$ and, moreover, $A(f^t)>0$ for all $t\ge m$. Hence for every $v\in V\Gamma$ and $t\ge 1$ we have $Wh_\Gamma(v,f)=Wh_\Gamma(v,f^t)$.
\end{rem}

\begin{lem}\label{lem:positive}
Let $\phi\in \Out(F_N)$ be an iwip and let $f:\Gamma\to\Gamma$ be a train-track representative of $\phi$. Then
\begin{enumerate}
\item The transition matrix $A(f)$ is irreducible and for each $e\in E\Gamma$ we have $|f^n(e)|\to\infty$ as $n\to\infty$.
\item There exists an integer $m\ge 1$ such that $A(f^m)>0$.
\end{enumerate}

\end{lem}
\begin{proof}
Part (1) is a straightforward corollary of the definitions, as observed, for example, on p. 5 of \cite{BH92}.

To see that (2) holds,  choose $s\ge 1$ such that every periodic vertex is fixed by $f^s$ and for every periodic edge $e$ of $\Gamma$ the path $f^s(e)$ begins with $e$. 
By part (1) we know that the length of every edge of $\Gamma$ goes to infinity under the iterations of $f$. Hence we can find a multiple $k$ of $s$ such that for every edge $e\in E\Gamma$ we have $|f^k(e)|\ge 2$. Put $g=f^k$. Thus $g:\Gamma\to\Gamma$ is a train-track representative of $\phi^k$. 

Now choose a periodic edge $e_0$ of $\Gamma$. Since $g(e_0)$ has length $\ge 2$ and starts with $e_0$, it follows that for every $n\ge 0$ the path $g^n(e_0)$ is a proper initial segment of $g^{n+1}(e_0)$. Let $\gamma=e_0,e_1,\dots, $ be a semi-infinite edge-path such that for all $n\ge 1$ $g^n(e_0)$ is an initial segment of $\gamma$. By construction we have $g(\gamma)=\gamma$. (That is why this $\gamma$ is sometimes called a \emph{combinatorial eigenray}, see~\cite{GJLL}). Let $\Gamma_0\subseteq \Gamma$ be the subgraph of $\Gamma$ obtained by taking the union of all the edges of $\gamma$ and their vertices.  By construction $g(\Gamma_0)\subseteq  \Gamma_0$ and hence $\Gamma_0=\Gamma$ since by assumption $\phi$ is an iwip and thus $\phi^k$ is irreducible. Thus there exists $t\ge 1$ such that $g^t(e_0)$ passes through every topological edge of $\Gamma$, and therefore, for all $n\ge t$ the path $g^n(e_0)$ passes through every topological edge of $\Gamma$. Applying the same argument to every periodic edge, we can find $t\ge 1$ such that for all $n\ge t$ and every periodic edge $e$ of $\Gamma$ the path $g^n(e)$ passes through every topological edge of $\Gamma$.

Since $E\Gamma$ is finite, there is an integer $b\ge 1$ such that for every edge $e\in E\Gamma$ the initial edge of $g^b(e)$ is periodic. Then for $m=b+t$ we have $A(g^m)=A(f^{km})>0$, as required.

\end{proof}

\begin{rem}\label{rem:m}
The proof of Lemma~\ref{lem:positive} can be straightforwardly modified to produce an algorithm that, given a train-track representative $f:\Gamma\to\Gamma$ of some $\phi\in \Out(F_N)$ such that $f$ satisfies condition (1) of Lemma~\ref{lem:positive}, decides whether or not there exists $m\ge 1$ such that $A(f^m)>0$, and if yes, produces such $m$. Namely, define $g=f^k$ exactly as in the proof of Lemma~\ref{lem:positive}. Then, given a periodic edge $e$, start iterating $g$ on $e$ until the first time we find $t\ge 1$ such that $g^{t+1}(e)$ passes through the same collection of topological edges of $\Gamma$ as does $g^t(e)$.  Let $\Gamma_0=\Gamma_0(e)$ be the subgraph of $\Gamma$ given by the union of edges of $g^t(e)$. By construction, we have $g(\Gamma_0)\subseteq \Gamma_0$. If $\Gamma_0\ne \Gamma$, then $\Gamma_0$ is a proper $f^k$-invariant subgraph of $\Gamma$ and hence there does not exist $m\ge 1$ such that $A(f^m)>0$.  If for every periodic edge $e$ we have $\Gamma_0(e)=\Gamma$, then we have found $t\ge 1$ such that for all $n\ge t$ and every periodic edge $e$ of $\Gamma$ the path $g^n(e)$ passes through every topological edge of $\Gamma$. Then, again as in the proof of Lemma~\ref{lem:positive}, we can find an integer $b\ge 1$ such that for every edge $e\in E\Gamma$ the initial edge of $g^b(e)$ is periodic. Then for $m=b+t$ we have $A(g^m)=A(f^{km})>0$.
\end{rem}

\section{Stable laminations}

In \cite{BFH97} Bestvina, Feighn and Handel defined the notion of a "stable lamination" associated to an iwip $\phi\in \Out(F_N)$.
A generalization of this notion for arbitrary automorphism plays a key role in the solution of the Tits Alternative for $\Out(F_N)$ by Bestvina, Feighn and Handel~\cite{BFH00,BFH05}. We need to state their definition of a "stable lamination" in a slightly more general context than that considered in \cite{BFH97}.

For the remainder of this section let $\phi\in \Out(F_N)$ be an outer automorphism (where $N\ge 2$) and let $f:\Gamma\to\Gamma$ be a train-track representative of $\phi$ such that for some $m\ge 1$ we have $A(f^m)>0$. (By a result of \cite{BH92} and Lemma~\ref{lem:positive} every iwip $\phi$ admits a train-track representative with the above property, and, moreover, every train-track representative of an iwip $\phi$ has this property.)

Note that the assumption on $f$ implies that $A(f^k)$ is irreducible for every $k\ge 1$ and, moreover, $A(f^k)>0$ for all $k\ge m$.

\begin{defn}[Stable lamination]
The \emph{stable lamination} $\Lambda(f)$ of $f$ consists of all the bi-infinite edge-paths
\[
\gamma= \dots e_{-1}, e_0, e_1, e_2, \dots 
\]
in $\Gamma$ with the following property:

For all $i\le j$, $i,j\in \mathbb Z$ there exist $n\ge 1$ and $e\in E\Gamma$ such that $e_i,\dots,e_j$ is a subpath of the path $f^n(e)$.
A path $\gamma$ as above is called a \emph{leaf} of $\Lambda(f)$.

Note that Remark~\ref{rem:positive}  implies that, under the assumptions on $f$ made in this section, for every $k\ge 1$ we have $\Lambda(f)=\Lambda(f^k)$.
\end{defn}

Let $H\le F_N$ be a nontrivial finitely generated subgroup. The \emph{$\Gamma$-Stallings core} $\Delta_H$ corresponding to $H$ (see \cite{Sta,KM} for details) is the smallest finite connected subgraph of the covering $\widehat\Gamma$ of $\Gamma$ corresponding to $H\le F_N$, such that the inclusion $\Delta_H\subseteq \widehat\Gamma$ is a homotopy equivalence. Note that $\Delta_H$ comes equipped with a canonical immersion $\Delta_H\to \Gamma$ obtained by the restriction of the covering map $\widehat\Gamma\to \Gamma$ to the subgraph $\Delta_H$.  By construction every vertex of $\Delta_H$ has degree $\ge 2$. Moreover, it is not hard to see that for every $w\in F_N$ we have $\Delta_H=\Delta_{wHw^{-1}}$.

We say that a nontrivial finitely generated subgroup $H\le F_N$ \emph{carries a leaf} of $\Lambda(f)$ if there exists a leaf $\gamma$ of $\Lambda(f)$ such that $\gamma$ lifts to a bi-infinite path in $\Delta_H$.

\section{Whitehead graphs and algorithmic decidability of being an iwip}

The following statement, based on the procedure of "blowing up" a train-track, is fairly well-known, and first appears, in somewhat more restricted context, in the proof of Proposition~4.5 in \cite{BH92}.  We present a sketch of the proof for completeness.

\begin{prop}\label{prop:BH92}
Let $N\ge 2$, $\phi\in \Out(F_N)$ and let $f:\Gamma\to\Gamma$ be an expanding train-track representative of $\phi$.
Suppose that there exists a vertex $u\in V\Gamma$ such that the Whitehead graph $Wh_\Gamma(u,f)$ is disconnected.
Then $\phi$ is reducible.

\end{prop}

\begin{proof}[Sketch of proof]
We construct a graph $\Gamma'$ and a graph-map $f':\Gamma'\to\Gamma'$ as follows. For each vertex $v$ of $\Gamma$ introduce a new vertex $v^\ast$ and $k$ vertices new $v_1,\dots, v_k$ where $k$ is the number of connected components of $Wh_\Gamma(v,f)$. We call $v^\ast$ a \emph{center-vertex} and the vertices $v_i$ \emph{sub-vertices}. The vertex set of $\Gamma'$ consists of the center-vertices and sub-vertices corresponding to all $v\in V\Gamma$.
The edge-set of $\Gamma'$ is a disjoint union of two sets of edges. First, every oriented edge $e$ of $\Gamma$ is also an edge of $\Gamma'$. For $e\in E\Gamma$ with $v=o(e)$ in $\Gamma$ we put $o(e)=v_i$ in $\Gamma'$ where $v_i$ is the sub-vertex coming from $v$ corresponding to the connected component of $Wh_\Gamma(u,f)$ containing $e$.
Second, for each $v\in V\Gamma$ with the corresponding sub-vertices $v_1,\dots, v_k$ we have an edge connecting $v^\ast$ and $v_i$ in $\Gamma'$. 
We call these latter types of edges of $\Gamma'$ \emph{sub-edges} corresponding to $v$.  Note that the graph $\Gamma'$ is connected but it may have degree-one vertices (namely, those  center-vertices $v^\ast$ such that $Wh_\Gamma(v,f)$ is connected).

We now define a map $f':\Gamma'\to \Gamma'$. For each vertex $v\in V\Gamma$ with $z=f(v)$ put $f'(v^\ast)=z^\ast$. 
Let $v_i$ be a sub-vertex corresponding to $v$ and $e$ is an edge of $\Gamma$ originating at $v$ and belonging to the connected component of $Wh_\Gamma(v,f)$ representing $v_i$. We put $f'(v_i)$ to be the sub-vertex at $z=f(v)$ corresponding to the initial edge $Df(e)$ of $f(e)$. It is easy to check that if two edges $e_1,e_2\in E\Gamma$ with origin $v$ are adjacent in $Wh_\Gamma(v,f)$ then the edges $Df(e_1)$ and $Df(e_2)$ are adjacent in $Wh_\Gamma(z,f)$.
It follows that for any edge $e\in E\Gamma$ the edge-path $f(e)$ in $\Gamma$ can also be viewed as an edge-path in $\Gamma'$ and we put $f'(e)=f(e)$.
Finally, if $e$ a sub-edge at $v$ joining $v^\ast$ and a sub-vertex $v_i$, and if $z=f(v)$, we put $f'(e)$ to be the sub-edge joining $z^\ast$ and the sub-vertex $f'(v_i)$.
A straightforward check shows that $f':\Gamma'\to \Gamma'$ is a continuous graph-map.
Moreover, contracting all the sub-edges in $\Gamma'$ to points is a homotopy equivalence between $\Gamma'$ and $\Gamma$. Thus $f':\Gamma'\to \Gamma'$ is a topological representative of $\phi$.

Let $\Delta$ be the subgraph of $\Gamma'$ given by the union of all the edges of $\Gamma$ and of their end-vertices in $\Gamma'$ (i.e. of all the sub-vertices). Thus, topologically, $\Delta$ is obtained from $\Gamma'$ by removing all the center-vertices and the interiors of all the sub-edges.

By construction we have $f'(\Delta)\subseteq \Delta$. The assumption that there exists a vertex $u\in V\Gamma$ such that the Whitehead graph $Wh_\Gamma(u,f)$ is disconnected implies that the inclusion $\Delta\subseteq \Gamma'$ is not a homotopy equivalence. Moreover, the graph $\Delta$ is not a forest. Indeed, by assumption $f$ is expanding. Choose an edge $e$ of $\Gamma$ and $n\ge 1$ such that the simplicial length of $f^n(e)$ is greater than the number of oriented edges in  $\Gamma$. Then $f^n(e)$ contains an edge subpath $\gamma$ such that $\gamma$ is a nontrivial simple circuit in $\Gamma$. Then, by definition of $\Gamma'$ and $\Delta$, $\gamma$ is also a circuit in $\Delta$. Thus $\Delta$ is not a forest. Since $\Delta$ is $f'$-invariant, homotopically nontrivial, and its inclusion in $\Gamma'$ is not a homotopy equivalence, we conclude that $\phi$ is reducible, as claimed.

\end{proof}

\begin{prop}\label{prop:BFH97}
Let $N\ge 2$, $\phi\in \Out(F_N)$ and let $f:\Gamma\to\Gamma$ be a train-track representative of $\phi$ such that $A(f)>0$. Suppose that for every $v\in V\Gamma$ the Whitehead graph $Wh_\Gamma(v,f)$ is connected.

Then there does not exist a finitely generated subgroup of infinite index in $F_N$ that carries a leaf of the lamination $\Lambda(f)$. 
\end{prop}
\begin{proof}[Sketch of proof]

The proof Proposition~2.4 in \cite{BFH97} and the proof of Lemma~2.1 in \cite{BFH97} on which Proposition~2.4 relies, work verbatim under the above assumptions.
The conclusion of Proposition~2.4 of \cite{BFH97} is exactly the conclusion that we need, namely that no f.g. subgroup of infinite index in $F_N$ carries a leaf of $\Lambda(f)$.
We provide a sketch of the proof, for completeness. 

Note that for any $k\ge 1$ we have $A(f^k)>0$ and $\Lambda(f^k)=\Lambda(f)$. Thus if needed, we can always replace $f$ by its positive power, and we will repeatedly do so below.

Suppose that a leaf of $\Lambda(f)$ is carried by a finitely generated infinite index subgroup $H\le F_N$.  First, by adding some edges, we complete $\Delta_H$  to a finite cover $\Gamma'$ of $\Gamma$. Note that since $H$ has infinite index in $F_N$, we really do need to add at least one new edge to get $\Gamma'$ from $\Gamma$.  By replacing $f$ by a power we may assume that $f$ fixes some vertex $v_0$ of $\Gamma$ and that $F_N=\pi_1(\Gamma,v_0)$.  Let $v_1$ be a vertex of $\Gamma'$ which projects to $v_1$ and let $H_1\le F_N=\pi_1(\Gamma,v_0)$ be the image in $\pi_1(\Gamma,v_0)$ of $\pi_1(\Gamma',v_1)$ under the covering map $\Gamma_1\to \Gamma$. Note that $[F_N:H_1]<\infty$.  Since $H_1$ has finite index in $F_N$, after replacing $f$ by its positive power, we may assume that $f_\#(H_1)=H_1$. Hence $f$ lifts to a map $f':\Gamma'\to\Gamma'$.
 Denote the covering map by $\pi: \Gamma'\to\Gamma$. By construction, $f'$ is a train-track map and for every $k\ge 1$ $(f')^k$ is a lift of $f^k$.  We may assume, after passing to powers, that for every $f'$-periodic edge $e'$ of $\Gamma'$ the path $f'(e')$ begins with $e'$, and that the same property holds for $f$.  Obviously, every turn in $\Gamma'$ taken by $f'$ projects to a turn in $\Gamma$ taken by $f$. 

{\bf Claim 1.} We claim that, after possibly replacing $f'$ by a further power,  if $a'b'$ be a reduced edge-path of length two in $\Gamma'$ projecting to a path $ab$ in $\Gamma$ such that the turn $a^{-1},b$ is taken by $f$ then the path $f'(a')$ contains the turn $(a')^{-1},b'$ and the path $f'(b')$ also contains the turn $(a')^{-1},b'$.

Indeed, let $a'b'$ be a reduced edge-path of length two in $\Gamma'$ projecting to a path $ab$ in $\Gamma$ such that the turn $a^{-1},b$ is taken by $f$. The assumption on $f$ implies that, after possibly passing to further powers, we have $f(a)=\dots ab\dots$ (here the dots represent nondegenerate edge-paths of positive simplicial length). This yields a fixed point $x$ of $f$ in the interior of $a$. Since $f$ is a homotopy equivalence, the map $f'$ permutes the finite set  $\pi^{-1}(x)$ in $\Gamma'$. Passing to a further power, we may assume that $f'$ actually fixes $\pi^{-1}(x)$ pointwise. Therefore we get a fixed point of $f'$ inside $a'$ and, using the fact that $\pi$ is a covering, we conclude that the path $f'(a')$ contains the turn $(a')^{-1},b'$. A similar argument shows that (again after possibly taking further powers), the path $f'(b')$ also contains the turn $(a')^{-1},b'$.  Thus Claim~1 is verified.

\bigskip

We now pass to an iterate of $f'$ for which the conclusion of Claim~1 holds, and replace $f$ by its corresponding iterate.
This implies, in particular, that a non-degenerate turn $\Gamma'$ is taken by $f'$ if and only if this turn is a lift to $ \Gamma'$ of a turn taken by $f$.Then for every vertex $v'$ of $\Gamma'$ projecting to a vertex $v$ in $\Gamma$ the Whitehead graph $Wh_{\Gamma'}(v',f')$ is exactly the lift of  $Wh_\Gamma(v,f)$, and, in particular, $Wh_{\Gamma'}(v',f')$ is connected.

{\bf Claim 2.} The matrix $A(f')$ is irreducible. 

Let $a',b'$ be arbitrary edges of $\Gamma'$. Consider the maximal subgraph $\Gamma''$ of $\Gamma'$ obtained as the union of all edges $c'$ admitting an edge-path $a'=e_0',\dots, e_n'=c'$ in $\Gamma'$ such that every turn contained in this path is taken by $f'$. Claim~1 above  now implies that for every edge $c'$ of $\Gamma'$ some $f'$-iterate of $a'$ passes through $c'$. We claim that $\Gamma''=\Gamma'$. If not, then there exists a vertex $v'$ of $\Gamma$ which is adjacent to both $\Gamma''$ and $\Gamma'\setminus \Gamma''$. The Whitehead graph $Wh_{\Gamma'}(v',f')$ is connected, and hence there is an $f'$-taken turn at $v'$ consisting of an edge of $\Gamma''$ and an edge of $\Gamma'\setminus \Gamma''$, contrary to maximality  of $\Gamma''$. Thus indeed $\Gamma''=\Gamma'$ and hence $b'\in E\Gamma'$. This means that some iterate of $a'$ under $f'$ passes through $b'$. Since $a',b'$ were arbitrary, it follows that $A(f')$ is irreducible, and Claim~2 is established.

Recall that we assumed that the statement of the proposition fails for a finitely generated subgroup of infinite index $H\le F_N$ , so that there exists a leaf $\gamma$ of $\Lambda(f)$ that lifts to $\Delta_H$. Choose an $f$-periodic edge $e$ in $\gamma$. Then for every $n\ge 1$ the path $f^n(e)$ lifts to a path $\alpha_n$ in $\Delta_H\subseteq \Gamma'$. Each $\alpha_n$ projects to $f^n(e)$ and starts with an $f'$-periodic edge $e_n'$. Since $\Gamma'$ is finite, we can find a sequence $n_i\to\infty$ as $i\to\infty$ and an $f'$-periodic edge $e'$  of $\Gamma'$ such that for all $i=1,2,\dots$ we have $e_{n_i}'=e'$, so that $\alpha_{n_i}$ starts with $e'$. Since $e'$ is $f'$-periodic and $f'(e')$ starts with $e'$, and since $f'$ is a lift of $f$, it follows that the path $(f')^{n_i}(e')=\alpha_{n_i}$ is contained in $\Delta_H$ for all $i\ge 1$. Since for every $s\le n_i$ $(f')^s(e')$ is an initial segment of $(f')^{n_i}(e')$, it follows that  for every $n\ge 1$ the path $(f')^n(e')$ is an edge-path in $\Delta_H$. Therefore for an edge $e''$  contained in $\Gamma_1\setminus\Delta_H$ there does not exist $n\ge 1$ such that $(f')^n(e')$ passes through $e''$.  This contradicts the fact that $A(f')$ is irreducible. 

\end{proof}

\begin{defn}[Clean train-track]
Let $f:\Gamma\to\Gamma$ be a train-track map. We say that $f$ is \emph{clean} if for some $m\ge 1$ we have $A(f^m)>0$ and if for every vertex $v$ of $\Gamma$ the Whitehead graph $Wh_\Gamma(f,v)$ is connected.
\end{defn}

\begin{prop}\label{prop:crit}
Let $N\ge 3$ and let $\phi\in \Out(F_N)$ be an atoroidal element. Then the following conditions are equivalent:

\begin{enumerate}
\item The automorphism $\phi$ is an iwip.

\item There exists a clean train-track representative $f:\Gamma\to\Gamma$ of $\phi$ and, moreover, every train-track representative of $\phi$ is clean.

\item There exists a clean train-track representative $f:\Gamma\to\Gamma$ of $\phi$.

\end{enumerate} 
\end{prop}

\begin{proof}

We first show that (1) implies (2). Thus suppose that $\phi$ is a atoroidal iwip. Then, as proved in \cite{BH92}, there exists a train-track representative of $\phi$. Let $f:\Gamma\to\Gamma$ be an arbitrary train-track representative of $\phi$. Since $\phi$ is an iwip, Lemma~\ref{lem:positive} implies that $A(f)$ is irreducible and that there exists $m\ge 1$ such that $A(f^m)>0$. Hence, by Remark~\ref{rem:positive}, for all $v\in V\Gamma$ and all $t\ge 1$ we have $Wh_\Gamma(f,v)=Wh_\Gamma(f^t,v)$.
Moreover, Proposition~\ref{prop:BH92} now implies that for every vertex $v$ of $\Gamma$ the Whitehead graph $Wh_\Gamma(f,v)=Wh_\Gamma(f^m,v)$ is connected. Thus $f$ is clean and condition (2) is verified.

It is obvious that (2) implies (3). It remains to show that (3) implies (1).  Thus suppose that there exists a clean train-track representative $f:\Gamma\to\Gamma$ of $\phi$.

We claim that $\phi$ is an iwip. 
Suppose not. Then $\phi^m$ is not an iwip either. Thus we may assume that $m=1$, so that  $A(f)>0$.

Then there exists a proper free factor $H$ of $F_N$ such that for some $k\ge 1$ we have $\phi^k([H])=[H]$. Let $\Delta_H$ be the $\Gamma$-Stallings core for $H$.
Choose a nontrivial element $h\in H$ and let $\gamma$ be an immersed circuit in $\Gamma$ representing the conjugacy class of $h$. Since by assumption $\phi$ is atoroidal, the cyclically tightened length of $f^n(\gamma)$ tends to $\infty$ as $n\to\infty$. Let $s$ be the simplicial length of $\gamma$, so that $\gamma=e_1\dots e_s$. 
 Let $\gamma_n$ be the immersed circuit in $\Gamma$ given by the cyclically tightened form of $f^{kn}(\gamma)$. We can obtain $\gamma_n$ by cyclic tightening of the path $f^{nk}(e_1)\dots f^{nk}(e_s)$. Thus $\gamma_n$ is a concatenation of $\le s$ segments, each of which is a subsegment of $f^{nk}(e)$ for some $e\in E\Gamma$. Since the simplicial length of $\gamma_n$ goes to infinity as $n\to\infty$, the length of at least one of these segments tends to infinity as $n\to\infty$.

By assumption $\gamma_n$ lifts to a circuit in $\Delta_H$. Hence there exists a sequence of segments $\alpha_n$ in $\Gamma$ such that each $\alpha_n$ lifts to a path in $\Delta_H$, such that the simplicial length of $\alpha_n$ goes to infinity as $n\to\infty$ and such that there are $e_n\in E\Gamma$ and  $t_n\ge 1$ with the property that $\alpha_n$ is a subpath of $f^{t_n}(e_n)$. Moreover, since $E\Gamma$ is finite, after passing to a subsequence we can even assume that $e_n=e\in E\Gamma$ for all $n\ge 1$. By a standard compactness argument, it follows that $H$ carries a leaf of $\Lambda(f)$, contrary to the conclusion of Proposition~\ref{prop:BFH97}.
Thus $\phi$ is an iwip, as claimed.

\end{proof}

\begin{rem}\label{rem:crit}
The assumption that $\phi$ be atoroidal in Proposition~\ref{prop:crit} is essential. One can construct $\phi\in \Out(F_N)$, coming from a pseudo-Anosov homeomorphism of a surface $S$ with $\ge 2$ punctures, such that there is a clean train-track $f:\Gamma\to\Gamma$ representing $\phi$. Then Proposition~\ref{prop:BFH97} still applies, and we do know that no leaf of $\Lambda(f)$ is carried by a finitely generated subgroup of infinite index in $F_N$. However, $\phi$ is not an iwip, since the peripheral curves around punctures in $S$ represent primitive elements in $F_N$ and thus generate cyclic subgroups that are periodic proper free factors of $F_N$.

A specific example of this kind is provided by Bestvina and Handel in Section~6.3 of \cite{BH95} and illustrated in Figure~33 on p. 139 of \cite{BH95}. In this example $S$ is a 5-pinctured sphere, so that $\pi_1(S)=F_4$, and $\phi$ is induced by a pseudo-Anosov homeomorphism of $S$ cyclically permuting the five punctures. The outer automorphism $\phi$ of $F_4=F(a,b,c,d)$ is represented by $\Phi\in \Aut(F_4)$ given by $\Phi(a)=b$, $\Phi(b)=c$, $\Phi(c)=da^{-1}$ and $\Phi(d)=d^{-1}c^{-1}$.   We can represent $\phi$ in the obvious way by a graph-map $f:\Gamma\to \Gamma$ where $\Gamma$ is the wedge of four loop-edges, corresponding to $a$, $b$, $c$, $d$, wedged at a single vertex $v$. Then, as observed in~\cite{BH95} and is easy to verify directly, $f$ is a train-track map with an irreducible transition matrix. A direct check shows that $Wh_\Gamma(v,f)$ is connected and that $A(f^6)>0$. Thus $f$ is a clean train-track representative of $\phi$. However, as noted above, $\phi$ is not an iwip. Thus the element $a\in F(a,b,c,d)$ in this example corresponds to a peripheral curve on $S$ and we see that $\Phi^5(a)=cd a d^{-1}c^{-1}$, so that $\phi^5$ preserves the conjugacy class of a proper free factor $\langle a\rangle$ of $F(a,b,c,d)$.  The fact that $\Phi^5(a)=cd a d^{-1}c^{-1}$ also explicitly demonstrates that $\phi$ is not atoroidal. Note also that in this example $\phi$ is irreducible but it is not an iwip, since $\phi^5$ is reducible. 

\end{rem}

\begin{thm}\label{thm:main}
There exists an algorithm that, given $N\ge 2$ and $\phi\in \Out(F_N)$ decides whether or not $\phi$ is an iwip.
\end{thm}
\begin{proof}
We first determine whether $\phi$ is atoroidal, as follows. Let $\Phi\in \Aut(F_N)$ be a representative of $\phi$ and put $G=F_N\rtimes_\Phi \langle t\rangle$ be the mapping torus group of $\Phi$. It is known, by a result of Brinkmann~\cite{Br}, that $\phi$ is atoroidal if and only if $G$ is word-hyperbolic. Thus we start running in parallel the following two procedures. The first is a partial algorithm, due to Papasoglu~\cite{Pa}, detecting hyperbolicity of $G$. The second procedure looks for $\phi$-periodic conjugacy classes of elements of $F_N$.
Eventually exactly one of these two processes will terminate and we will know whether or not $\phi$ is atoroidal.

{\bf Case 1.} Suppose first that $\phi$ turns out to be atoroidal (and hence $N\ge 3$). 

We then run an algorithm of Bestvina-Handel~\cite{BH92} which tries to construct a train-track representative of $\phi$.  As proved in \cite{BH92}, this algorithm always terminates and either produces a train-track representative of $\phi$ with an irreducible transition matrix or finds a reduction for $\phi$, thus showing that $\phi$ is reducible. If the latter happens, we conclude that $\phi$ is not an iwip. Suppose now that the  former happens and we have found a train-track representative $f:\Gamma\to\Gamma$ of $\phi$ with irreducible $A(f)$.
We first check if it is true that for every edge $e$ of $\Gamma$ there exists $t\ge 1$ such that $|f(e)|\ge 2$. If not, we conclude, by Lemma~\ref{lem:positive}, that $\phi$ is not an iwip. If yes, we then check, e.g. using the algorithm from Remark~\ref{rem:m},  if there exists an integer $m\ge 1$ such that $A(f^m)=(A(f))^m>0$.
If no such $m\ge 1$ exists, we conclude, again by Lemma~\ref{lem:positive},  that $\phi$ is not an iwip. Suppose now we have found $m\ge 1$ such that $A(f^m)>0$.  We then check if it is true that every vertex of $\Gamma$ has a connected Whitehead graph $Wh_\Gamma(v,f)$. If not, then we conclude that $\phi$ is not an iwip, by Proposition~\ref{prop:BH92}. If yes,  then $f$ is clean and we conclude that $\phi$ is an iwip, by  Proposition~\ref{prop:crit}. Thus for an atoroidal $\phi$ we can indeed algorithmically determine whether or not $\phi$ is an iwip.

{\bf Case 2.} Suppose now that $\phi$ turned out to be non-atoroidal. Then Proposition~4.5 of \cite{BH92} implies that $\phi$ is an iwip if and only if $\phi$ is induced by a pseudo-Anosov homeomorphism of a compact surface $S$ with a single boundary component. Thus either $\phi$ has a periodic conjugacy class of a proper free factor of $F_N$ or $\phi$ is induced by a pseudo-Anosov  of a compact surface $S$ with a single boundary component.  

We now start running in parallel the following two processes.

The first process looks for a periodic conjugacy class of a proper free factor of $F_N$: we start enumerating all the proper free factors $H_1, H_2,\dots $ of $F_N$ and for each $H_i$ we start listing its images $\phi(H_i), \phi^2(H_i), \phi^3(H_i), \dots $ and check if $\phi^j([H_i])=[H_i]$. The process terminates if we find $i,j$ such that $\phi^j([H_i])=[H_i]$.

The second process looks for the relization of $\phi$ as a a pseudo-Anosov homeomorphism of a compact surface $S$ as above. Note that if $g$ belongs to the mapping class group $Mod(S)$ of $S$ and if $\alpha_1,\alpha_2: F_N\to \pi_1(S)$ are two isomorphisms, then the elements of $\Out(F_N)$ corresponding to $g$ via $\alpha_1$ and $\alpha_2$ are related by a conjugation in $\Out(F_N)$.  Thus, in order to account for all possible realizations of $\phi$ of the above type, do the following. Depending on the rank $N$ of $F_N$, there are either exactly one (non-orientable) or exactly two (one orientable and one non-orientable) topological types of compact connected surfaces $S$ with one boundary component and with $\pi_1(S)$ free of rank $N$. For each of these  choices of $S$ we fix an isomorphism $\alpha:F_N\to \pi_1(S)$. Then start enumerating all the elements $g_1,g_2,\dots$ of $Mod(S)$, and, for each such $g_i$, start enumerating all the $\Out(F_N)$-conjugates $\psi_{ij}$, $j=1,2,\dots$ of the element of $\Out(F_N)$ corresponding to $g_i$ via $\alpha$. Then for each $\psi_{ij}$ check if $\psi_{ij}=\phi$ in $\Out(F_N)$. If not, continue the enumeration of all the the $\psi_{ij}$, and if yes, use the algorithm from~\cite{BH95} to decide whether  or not $g_i$ is pseudo-Anosov (see the paper of Brinkmann~\cite{Br1} for the details about how this Bestvina-Handel algorithm works for compact surfaces with one boundary component). If $g_{i}$ is pseudo-Anosov, we terminate the process; otherwise, we continue the diagonal enumeration of all the $\psi_{ij}$.

Eventually exactly one of these two processes will terminate.  If the first process terminates, we conclude tha $\phi$ is not an iwip. If the second process terminates, we conclude that $\phi$ is an iwip.

\end{proof}

\section{Further developments}

After this paper was written,  the result of  Proposition~\ref{prop:crit} was improved by Dowdall, Kapovich and Leininger~\cite{DKL}.
Let $N\ge 2$, $\phi\in \Out(F_N)$ and let $f:\Gamma\to\Gamma$ be a train-track representative of $\phi$. We say that $f$ is \emph{weakly clean} if  $A(f)$ is irreducible, $f$ is expanding and if for every vertex $v$ of $\Gamma$ the Whitehead graph $Wh_\Gamma(f,v)$ is connected. 

Proposition~5.2 of \cite{DKL} proves:
\begin{prop}\label{prop:DKL}
Let $N\ge 2$, $\phi\in \Out(F_N)$ and let $f:\Gamma\to\Gamma$ be a weakly clean train-track representative of $\phi$. 
Then there exists $k\ge 1$ such that $A(f^k)>0$ (and hence  $f$ is clean).
\end{prop}

Therefore, by  Proposition~\ref{prop:crit}, if $\phi\in \Out(F_N)$ is atoroidal and if  $f:\Gamma\to \Gamma$ is a weakly clean train-track representative of $\phi$ then $\phi$ is an iwip.
Moreover, as observed in \cite{DKL}, Proposition~\ref{prop:DKL} can be used to show that for atoroidal elements of $\Out(F_N)$ being irreducible is equivalent to being an iwip.

Thus Proposition~5.6 of \cite{DKL} proves:
\begin{cor}\label{cor:DKL}
Let $N\ge 3$ and let $\phi\in \Out(F_N)$ be an atoroidal element. 

Then $\phi$ is irreducible if and only if $\phi$ is an iwip.  
\end{cor}
\begin{proof}
Clearly, if $\phi$ is an iwip then $\phi$ is irreducible.

Thus let $\phi\in \Out(F_N)$ be an atoroidal irreducible element.  Then, by a result of Bestvina-Handel, there exists a train-track representative $f:\Gamma\to\Gamma$ of $\phi$ such that  the transition matrix of $A(f)$ is irreducible. Since $\phi$ is atoroidal, it follows that $A(f)$ is not a permutation matrix, and therefore the train-track map $f$ is expanding. Since by assumption $\phi$ is irreducible,  Proposition~\ref{prop:BH92} implies that for every vertex $v$ of $\Gamma$ the Whitehead graph $Wh_\Gamma(f,v)$ is connected. Hence, by Proposition~\ref{prop:DKL}, $f$ is clean. Therefore, by Proposition~\ref{prop:crit}, $\phi$ is an iwip, as required.

\end{proof}

\end{document}